\documentclass[a4paper]{article}
\usepackage{amsmath,amssymb,amsthm}%公式环境,数学符号,数学定理
\usepackage{amsopn}%自定义算子
\usepackage{xcolor}%颜色宏包

\usepackage[final]{showkeys}%书签名显示
\usepackage[right,mathlines]{lineno}%行号\linenumbers
\usepackage[colorlinks,linkcolor=red,anchorcolor=blue,citecolor=green]{hyperref}% 链接宏包
\usepackage{cleveref}%cref

\DeclareMathOperator{\Tr}{Tr}

\newtheorem{theorem}{Theorem}[section]
\newtheorem{definition}[theorem]{Definition}
\newtheorem{lemma}[theorem]{Lemma}

\newtheorem{corollary}[theorem]{Corollary}
\newtheorem{remark}[theorem]{Remark}
\newtheorem{example}[theorem]{Example}
\newcounter{Stepnumber}

\begin{document}
%题名信息
\title{Formula for the IDS of Periodic Jacobi Matrices}
\author{Liangping Qi\thanks{School of Science and Technology, Tianjin University of Finance and Economics, Tianjin, P. R. China. Email: lqi@tjufe.edu.cn.}}
\date{}
\maketitle

\begin{abstract}
We prove that on the spectrum the integrated density of states (IDS, for short) of periodic Jacobi matrices is related to the discriminant. The method is to count the number of generalized zeros of Bloch wave solutions.
\end{abstract}



\section{Introduction}

Jacobi matrices acting on $l^2(\mathbb{Z})$
\begin{equation}\label{eq:H}
  (Hu)(n)=u(n+1)+u(n-1)+V(n)u(n), n\in \mathbb{Z},
\end{equation}
and the continuum analog, $-\Delta+q$, on $L^2(\mathbb{R})$, have subtle and fascinating spectral properties. The IDS, $k^H(E)$, for such operators with almost periodic potentials is studied in \cite{AS83}. In this almost periodic regime, the spectrum of $H$ is given by the points of non-constancy of $k^H$, i.e., $\sigma(H)=\{E; k^H(E+\epsilon)-k^H(E-\epsilon)>0, \forall \epsilon>0\}$. Moreover, $k^H(E)$ is proved to lie in the frequency module for real $E$ in the resolvent set \cite{DS83,JM82}. This yields a labelling of the gaps of the spectrum.

For $p$-periodic $V$, $k^H(E)$ has the value $j/p$ on the $j$-th gap \cite{BS82}. The expression of $k^H(E)$ on $\sigma(H)$ is only known to be complicate. We shall mainly reveal the following formula.

\begin{theorem}\label{th:IDS}
Suppose that $V: \mathbb{Z}\rightarrow \mathbb{R}$ is a $p$-periodic potential, $B_j=[a_j, b_j]$, $j=1$, $\ldots$, $p$, are the $p$ bands of $\sigma(H)$ with
\begin{equation*}
  a_1<b_1\leq a_2<b_2\leq \cdots \leq a_p<b_p,
\end{equation*}
$G_0=(-\infty, a_1)$ and $G_p=(b_p, \infty)$ are unbounded gaps, $G_j=(b_j, a_{j+1})$ or $G_j=\{b_j\}=\{a_{j+1}\}$ is the open or collapsed gap which separate $B_j$ and $B_{j+1}$ for $j=1$, $\ldots$, $p-1$. Then
\begin{equation}\label{eq:IDS}
k^H(E)=
\begin{cases}
\dfrac{j}{p},\quad &E\in G_j,\ j=0, \ldots, p,\\
\dfrac{j}{p}-\dfrac{\vartheta}{\pi},\quad &E\in B_j, 1\leq j\leq p, p-j=2m, m\in \mathbb{Z},\\
\dfrac{j-1}{p}+\dfrac{\vartheta}{\pi},\quad &E\in B_j, 1\leq j\leq p, p-j=2m-1, m\in \mathbb{Z},
\end{cases}
\end{equation}
where $\vartheta\in [0, \pi/p]$ is determined by $D(E)=2\cos p\vartheta$ for $E\in \sigma(H)$.
\end{theorem}

\section{Proofs}

Denote by $M_E(n)$ the transfer matrix associated with the eigenvalue equation
\begin{equation}\label{eq:eigenH}
  Hu=Eu,
\end{equation}
where $E\in \mathbb{C}$ and $n\in \mathbb{Z}$. Let $D(E)=\Tr M_E(p)$ be the discriminant. Denote by $\# S$ the number of elements in a set $S$. Set $J=-H$.

\begin{definition}[\cite{DS83}]\label{df:Nk}
Let $E\in \mathbb{R}$, $A=H$ or $J$, $u^A$ be the solution to the eigenvalue equation $Au=Eu$ with initial condition $u(0)=\cos\theta$, $u(1)=\sin\theta$, and $A_L$ the restriction of $A$ to $[1, L-1]\cap \mathbb{Z}$ with boundary conditions $u(0)/u(1)=\cot\theta$, $u(L)=0$. Define
\begin{align*}
N_L^A(E)&=\#\{n\in (1, L]\cap \mathbb{Z}; u^A(n)=0~\text{or}~u^A(n-1)u^A(n)<0\},\\
k_L^A(E)&=\frac{1}{L-1}\#\{\lambda\in\sigma(A_L); \lambda\leq E\},
\end{align*}
where the dependence on $\theta$ is left implicit. That is, $N_L^A(E)$ is the number of generalized zeros \cite[p. 321]{Ela05} of $u^A$ in $(1, L]\cap \mathbb{Z}$, and $k_L^A(E)$ the IDS for $A_L$.
\end{definition}

It is well-known that there exists the limit $\lim_{L\rightarrow\infty}k_L^A(E)=:k^A(E)$, called the integrated density of states (IDS), independent of $\theta$, for $A=H$ and $J$. The following result is basic in computing the IDS. The formula for $k_L^J$ is proved in \cite{DS83}. The proof of the one for $k_L^H$ is similar. So we omit it.
\begin{lemma}\label{lem:NkL}
$k_L^J(E)=\dfrac{N_L^J(E)}{L-1}$, $k_L^H(E)=1-\dfrac{N_L^H(E)}{L-1}$.
\end{lemma}

For energies $E$ in the spectrum, the existence of Bloch wave solutions is the key observation in determining the IDS.
\begin{lemma}\label{lem:Bloch}
Suppose that $E\in \mathbb{R}$ and $\vartheta\in (0, \pi/p)$, then $D(E)=2\cos p\vartheta$ if and only if \eqref{eq:eigenH} has a nontrivial complex solution $u$ such that
\begin{equation}\label{eq:Bloch}
  u(n+p)=e^{ip\vartheta}u(n),\quad \forall n\in \mathbb{Z}.
\end{equation}
Moreover, $u_1:=\Re u$ and $u_2:=\Im u$ are real linearly independent solutions to \eqref{eq:eigenH}.
\end{lemma}
\begin{proof}
The existence of such a nontrivial solution satisfying \eqref{eq:Bloch} is well-known, e.g., \cite{Las92}. Since $\vartheta\in (0, \pi/p)$, \eqref{eq:Bloch} implies that $u$ is not real. Because $E\in \mathbb{R}$, $u_1$ and $u_2$ are real solutions to \eqref{eq:eigenH}. If there is a real constant $c$ with $u_1=cu_2$, a direct calculation shows that
\begin{align*}
(c+i)u_2(n+p)&=(c+i)e^{ip\vartheta}u_2(n),\\
u_2(n+p)&=e^{ip\vartheta}u_2(n)
\end{align*}
for all $n\in \mathbb{Z}$. This contradicts the fact that $u_2$ is a real sequence.
\end{proof}

\begin{proof}[Proof of \Cref{th:IDS}]
Starting from \Cref{lem:NkL}, we shall reduce the problem of counting the number of generalized zeros to the one of estimating the argument of a continuous function, which is easier to deal with.

Let $E$ be in the interior of $\sigma(H)$ with $D(E)=2\cos p\vartheta$ for some $\vartheta\in (0, \pi/p)$, $u$, $u_1$ and $u_2$ be the sequences as in \Cref{lem:Bloch}. Since $u_1$ and $u_2$ are linearly independent solutions of \eqref{eq:eigenH}, the Wronskian of $u_1$ and $u_2$ is a nonzero constant which is denoted by $W_0$. Therefore, for all $n\in \mathbb{Z}$,
\begin{equation}\label{eq:W}
\begin{split}
0\neq W_0&=
\begin{vmatrix}
  u_1(n+1) & u_2(n+1) \\
  u_1(n) & u_2(n) \\
\end{vmatrix}
=
\begin{vmatrix}
  u_1(n+1) & u_1(n) \\
  u_2(n+1) & u_2(n) \\
\end{vmatrix}
\\
&=|u(n+1)|\cdot |u(n)|\cdot \sin\langle u(n+1), u(u)\rangle,
\end{split}
\end{equation}
where $\langle u(n+1), u(u)\rangle\in (-\pi, \pi)$ is the angle of rotation from $u(n+1)$ to $u(n)$. This observation is crucial in proving \eqref{eq:IDS}. In the following we would identify $\mathbb{R}^2$ with $\mathbb{C}$. We shall consider two cases of the sign of $W_0$.

Case 1. $W_0<0$. $u(n)$ would rotate around the origin anticlockwise in $\mathbb{R}^2$ as $n$ increases. The angle of rotation from $u(n)$ to $u(n+1)$, $\langle u(u), u(n+1)\rangle$ will belong to $(0,\pi)$ for all $n\in \mathbb{Z}$. Let $\tilde{u}(t)$ be the linear interpolation of the sequence $\{u(n)\}_{n\in \mathbb{Z}}$, that  is,
\begin{align*}
\tilde{u}(t)&=\tilde{u}_1(t)+i\tilde{u}_2(t)\\
&=(n+1-t)u(n)+(t-n)u(n+1),\quad n\leq t\leq n+1,\ n\in \mathbb{Z}.
\end{align*}
By \eqref{eq:W}, $\tilde{u}(t)\neq0$ for all $t\in \mathbb{R}$, otherwise $u_1$ and $u_2$ would be linearly dependent. Therefore,
\begin{align*}
\tilde{N}_{L, j}(E):&=\#\{t\in (1, L]; \tilde{u}_j(t)=0\}\\
&=\#\{n\in (1, L]\cap \mathbb{Z}; u_j(n)=0~\text{or}~u_j(n-1)u_j(n)<0\}=:N_{L, j}^H(E)
\end{align*}
for $j=1$, $2$ and
\begin{equation}\label{eq:NL}
\begin{split}
\tilde{N}_L(E):&=\#\{t\in (1, L]; \tilde{u}_1(t)\cdot \tilde{u}_2(t)=0\}\\
&=\tilde{N}_{L, 1}(E)+\tilde{N}_{L, 2}(E)=N_{L, 1}^H(E)+N_{L, 2}^H(E).
\end{split}
\end{equation}
Let $\arg \tilde{u}(t)$ be a continuous branch of the argument of $\tilde{u}(t)$. \eqref{eq:W} yields that $\tilde{u}(t)$ would cross the $u_1$- and $u_2$-axis alternatively. The angle of rotation from one crossing to the subsequent one is exactly $\pi/2$. Moreover, if $\arg \tilde{u}(t_2)-\arg \tilde{u}(t_1)=\pi/2$, there is a unique time $t_0\in (t_1, t_2]$ when $\tilde{u}(t)$ crosses the axis at $t_0$. Consequently,
\begin{equation}\label{eq:-pi/2}
\Big|\arg \tilde{u}(L)-\arg \tilde{u}(1)-\frac{\pi}{2}\tilde{N}_L(E)\Big|<\frac{\pi}{2}
\end{equation}
and
\begin{equation}\label{eq:arg-k}
\lim_{L\rightarrow\infty}\frac{\arg \tilde{u}(L)}{L-1}=\frac{\pi}{2}\lim_{L\rightarrow\infty}
\frac{N_{L, 1}^H(E)+N_{L, 2}^H(E)}{L-1}=\pi[1-k^H(E)]
\end{equation}
by \eqref{eq:-pi/2}, \eqref{eq:NL} and \Cref{lem:NkL}.

By \eqref{eq:Bloch}, there exists an $m_n\in \mathbb{Z}$ such that
\begin{equation*}
\arg \tilde{u}(n+p)-\arg \tilde{u}(n)=p\vartheta+2m_n\pi
\end{equation*}
for all $n\in \mathbb{Z}$. From \eqref{eq:W} it follows that $0<2m_n<p$ and $\arg \tilde{u}(n+1)-\arg \tilde{u}(n)$, $\arg \tilde{u}(n+p+1)-\arg \tilde{u}(n+p)\in (0, \pi)$. Therefore,
\begin{equation*}
|[\arg \tilde{u}(n+p+1)-\arg \tilde{u}(n+1)]-[\arg \tilde{u}(n+p)-\arg \tilde{u}(n)]|<\pi,
\end{equation*}
which implies $m_n=m_{n+1}$. Hence $m:=m_n$ is independent of $n\in \mathbb{Z}$. By iteration one arrives at
\begin{equation*}
\arg \tilde{u}(lp+1)-\arg \tilde{u}(1)=lp\vartheta+2lm\pi
\end{equation*}
for all $l\in \mathbb{Z}_+$. A straightforward computation shows that
\begin{equation*}
\lim_{l\rightarrow\infty}\frac{\arg \tilde{u}(lp+1)}{lp}=\vartheta+\frac{m}{p}2\pi
\end{equation*}
and by \eqref{eq:arg-k},
\begin{equation}\label{eq:-k}
k^H(E)=1-\frac{\vartheta}{\pi}-\frac{2m}{p}.
\end{equation}
Furthermore, since $k^H(E)$ is continuous in $E$ and $D(E)=2\cos p\vartheta$, $m$ is independent of energies in the same band.

Case 2. $W_0>0$. $u(n)$ would rotate around the origin clockwise in $\mathbb{R}^2$ as $n$ increases. Define the continuous function $\tilde{u}(t)$ as above. \eqref{eq:-pi/2} and \eqref{eq:arg-k} take respectively the form of
\begin{align*}
&\Big|\arg \tilde{u}(L)-\arg \tilde{u}(1)+\frac{\pi}{2}\tilde{N}_L(E)\Big|<\frac{\pi}{2},\\
&-\lim_{L\rightarrow\infty}\frac{\arg \tilde{u}(L)}{L-1}=\pi[1-k^H(E)].
\end{align*}
In this case, there is an $m\in \mathbb{Z}$ independent of $n$ and energies in the same band such that $-p<2m<0$ and
\begin{equation*}
\arg \tilde{u}(n+p)-\arg \tilde{u}(n)=p\vartheta+2m\pi
\end{equation*}
for all $n\in \mathbb{Z}$. Thus
\begin{equation*}
-\lim_{l\rightarrow\infty}\frac{\arg \tilde{u}(lp+1)}{lp}=-\Big(\vartheta+\frac{m}{p}2\pi\Big)
\end{equation*}
and
\begin{equation}\label{eq:+k}
k^H(E)=1+\frac{\vartheta}{\pi}+\frac{2m}{p}.
\end{equation}

At last, we make use of properties of the IDS to obtain the final conclusion. It is clear that
\begin{equation*}
k^H(E)=
\begin{cases}
1,\quad &E>>1,\\
0,\quad &-E>>1.
\end{cases}
\end{equation*}
Since each band $B_j$ is parameterized by $D(E)=2\cos p\vartheta$ with $\vartheta\in [0, \pi/p]$, from \eqref{eq:-k} and \eqref{eq:+k} it follows that
\begin{equation*}
k^H(b_j)-k^H(a_j)=\frac{1}{p}
\end{equation*}
for all $j=1$, $\ldots$, $p$. The continuity and the constancy on gaps of $k^H$ yield
\begin{equation}\label{eq:kG}
k^H(E)=\frac{j}{p},\quad E\in G_j,\ j=0, \ldots, p.
\end{equation}
By \eqref{eq:kG}, $k^H(E)\in [(j-1)/p, j/p]$ for $E\in B_j$ and $j=1$, $\ldots$, $p$. Because $\vartheta/\pi\in [0, 1/p]$, we must have $1-2m/p=j/p$ and $1+2m/p=(j-1)/p$ in \eqref{eq:-k} and \eqref{eq:+k}, respectively. From the evenness of $2m$ it follows that
\begin{equation*}
k^H(E)=
\begin{cases}
\dfrac{j}{p}-\dfrac{\vartheta}{\pi},\quad &E\in B_j, 1\leq j\leq p, p-j=2m, m\in \mathbb{Z},\\
\dfrac{j-1}{p}+\dfrac{\vartheta}{\pi},\quad &E\in B_j, 1\leq j\leq p, p-j=2m-1, m\in \mathbb{Z}.
\end{cases}
\end{equation*}
\end{proof}

\begin{remark}
Since $Hu=Eu \Leftrightarrow Ju=-Eu$, it is clear that $N_L^H(E)=N_L^J(-E)$, which yields
\begin{equation}\label{eq:kHJ}
  k^H(E)+k^J(-E)=1
\end{equation}
by \Cref{lem:NkL} and taking limits. From \eqref{eq:kHJ} it follows that \eqref{eq:IDS} for $J$ takes the form of
\begin{equation}\label{eq:IDSJ}
k^J(E)=
\begin{cases}
\dfrac{j}{p},\quad &E\in G_j,\ j=0, \ldots, p,\\
\dfrac{j}{p}-\dfrac{\vartheta}{\pi},\quad &E\in B_j, 1\leq j\leq p, j=2m, m\in \mathbb{Z},\\
\dfrac{j-1}{p}+\dfrac{\vartheta}{\pi},\quad &E\in B_j, 1\leq j\leq p, j=2m+1, m\in \mathbb{Z},
\end{cases}
\end{equation}
where $\vartheta\in [0, \pi/p]$ is determined by $\Tr M_E^J(p)=2\cos p\vartheta$ for $E\in \sigma(H)$, and $M_E^J(n)$ is the transfer matrix associated with the equation $Ju=Eu$.
\iffalse
A detailed proof.
As for $J$, since $\Tr M_E^J(p)=D(-E)$ and $\sigma(J)=-\sigma(H)$, we obtain $\sigma(J)=[\Tr M_\cdot^J(p)]^{-1}([-2, 2])$. Furthermore, \Cref{lem:Bloch} is true with $H$ replaced by $J$. Indeed, $\Tr M_E^J(p)=2\cos p\vartheta$ is equivalent to $e^{ip\vartheta}\in \sigma(M_E^J(p))$. An eigenvector $(u(1), u(0))^T$ of $M_E^J(p)$ belonging to $e^{ip\vartheta}$ is an initial value of a nontrivial solution $u$ of $Ju=Eu$ satisfying \eqref{eq:Bloch}. With the same argument as above, \eqref{eq:-k} and \eqref{eq:+k} take the form of
\begin{align}
  k^J(E) & =\frac{1}{\pi}\lim_{l\rightarrow\infty} \frac{\arg\tilde{u}(lp)}{lp}
  =\frac{\vartheta}{\pi}+\frac{2m}{p},\label{eq:-kJ} \\
  k^J(E) & =-\frac{1}{\pi}\lim_{l\rightarrow\infty} \frac{\arg\tilde{u}(lp)}{lp}
  =-\frac{\vartheta}{\pi}-\frac{2m}{p},\label{eq:+kJ}
\end{align}
respectively. In this case, \eqref{eq:kG} holds for $J$ and $k^J(E)\in [(j-1)/p, j/p]$ for $E\in G_j$ and $j=1$, $\ldots$, $p$. Therefore, $2m/p=(j-1)/p$ and $-2m/p=j/p$ in \eqref{eq:-kJ} and \eqref{eq:+kJ}, respectively. From the evenness of $2m$ we get
\begin{equation*}
k^J(E)=
\begin{cases}
\frac{j}{p}-\frac{\vartheta}{\pi},\quad &E\in B_j, 1\leq j\leq p, j=2m, m\in \mathbb{Z},\\
\frac{j-1}{p}+\frac{\vartheta}{\pi},\quad &E\in B_j, 1\leq j\leq p, j=2m+1, m\in \mathbb{Z}.
\end{cases}
\end{equation*}
\fi
\end{remark}

\begin{corollary}\label{cor:IDS1}
Suppose that $G_j$ is a gap defined in \Cref{th:IDS}, $E\in \partial G_j$. Then there is a real Bloch wave solution to \eqref{eq:eigenH} such that $N_{p+1}^H(E)=p-j$.
\end{corollary}
\begin{proof}
$E\in \partial G_j$ yields $D(E)=\pm2$ and $\vartheta=0, \pi/p$. By taking real and imaginary parts, the existence of a nontrivial Bloch wave solution implies the existence of a nontrivial real one, which is denoted by $u_E^H$. From $u_E^H(n+p)=\pm u_E^H(n)$, $n\in \mathbb{Z}$, it follows that $N_{lp+1}^H(E)=lN_{p+1}^H(E)$ for all $l\in \mathbb{Z}_+$. Therefore,
\begin{equation*}
  k^H(E)=1-\frac{N_{p+1}^H(E)}{p}=\frac{j}{p}
\end{equation*}
by \Cref{lem:NkL} and \Cref{th:IDS}. Hence $N_{p+1}^H(E)=p-j$.
\end{proof}

\Cref{th:IDS} reveals a new proof of the following classical result.
\begin{corollary}\label{cor:IDS2}
Suppose that $E$ is the $j$-th eigenvalue of the restriction of $H$ to $[2, p]\cap \mathbb{Z}$ with Dirichlet boundary condition $u(1)=u(p+1)=0$, $u_E^H$ is the corresponding eigenvector. Then $N_{p+1}^H(E)=p-j$.
\end{corollary}
\begin{proof}
It is known that the eigenvalues of the boundary value problem above are simple and separate the bands of $\sigma(H)$ \cite{Las92}. So $E\in \overline{G_j}$. Consequently, $(u(1), \ldots, u(p+1))$ and $(u(lp+1), \ldots, u(lp+p+1))$ are linearly dependent for all $l\in \mathbb{Z_+}$. Hence $N_{lp+1}^H(E)=lN_{p+1}^H(E)$ and $N_{p+1}^H(E)=p-j$ by \Cref{lem:NkL} and \Cref{th:IDS}.
\end{proof}

\iffalse
\begin{example}
Consider $H$ and $J$ with $V=0$. Thus $V$ is $1$-periodic. It is well-known that $\sigma(H)=\sigma(J)=[-2, 2]$. So both $H$ and $J$ have a band $B_1=[-2, 2]$ and two unbounded gaps $G_0=(-\infty, -2]$ and $G_1=[2, \infty)$. From
\begin{align*}
&M_E(1)=
\begin{pmatrix}
  E & -1 \\
  1 & 0 \\
\end{pmatrix},\quad D(E)=E,\\
&M_E^J(1)=
\begin{pmatrix}
  -E & -1 \\
  1 & 0 \\
\end{pmatrix},\quad \Tr M_E^J(1)=-E
\end{align*}
it follows that $B_1$ is parameterized by $E=2\cos\vartheta$ and $-E=2\cos\vartheta$ respectively for $H$ and $J$, where $\vartheta\in [0, \pi]$. Since $p-j=1-1$ is even and $j=1$ is odd, \eqref{eq:IDS} and \eqref{eq:IDSJ} yield
\begin{equation*}
k^H(E)=1-\frac{\vartheta}{\pi},\quad k^J(E)=\frac{\vartheta}{\pi}
\end{equation*}
for $E\in B_1$. This coincide with the monotonicity of the ids on the spectrum.
\end{example}

\fi

\section{Acknowledgments}

Supported by The National Nature Science Foundation of China (No. 12001397), The Nature Science Foundation of Tianjin (No. 20JCQNJC00970).
%The authors thank the anonymous referee for valuable comments.

\end{document}